\documentclass[a4paper, 11pt]{amsart}
\usepackage{amsfonts, amsmath, amssymb, amsthm,mathtools, url,dsfont}
\usepackage{graphicx}
\usepackage[english] {babel}
\usepackage{enumitem}
\usepackage{array}
\usepackage{cellspace}
\usepackage{color}
\newtheorem{theorem}{Theorem}

\newtheorem{lemma}{Lemma}

\newcommand{\Q}{\mathbb{Q}}

\newcommand{\ve}{\varepsilon}

\usepackage{subcaption}
\usepackage{graphicx}

\begin{document}

\begin{abstract}
A deterministic sequence of real numbers in the unit interval is called \emph{equidistributed} if its empirical distribution converges to the uniform distribution. Furthermore, the limit distribution of the pair correlation statistics of a sequence is called \emph{Poissonian} if the number of pairs $x_k,x_l \in (x_n)_{1 \leq n \leq N}$ which are within distance $s/N$ of each other is asymptotically $\sim 2sN$. A randomly generated sequence has both of these properties, almost surely. There seems to be a vague sense that having Poissonian pair correlations is a ``finer'' property than being equidistributed. In this note we prove that this really is the case, in a precise mathematical sense: a sequence whose asymptotic distribution of pair correlations is Poissonian must necessarily be equidistributed. Furthermore, for sequences which are not equidistributed we prove that the square-integral of the asymptotic density of the sequence gives a lower bound for the asymptotic distribution of the pair correlations.
\end{abstract}

\title[]{Pair correlations and equidistribution}

\author{Christoph Aistleitner} 
\address{Christoph Aistleitner, Institute of Analysis and Number Theory, TU Graz, Austria}
\email{aistleitner@math.tugraz.at}

\author{Thomas Lachmann} 
\address{Thomas Lachmann, Institute of Analysis and Number Theory, TU Graz, Austria}
\email{lachmann@math.tugraz.at}

\author{Florian Pausinger}
\address{Florian Pausinger, Chair of Geometry and Visualization, TU Munich, Germany}
\email{florian.pausinger@ma.tum.de}

\thanks{The first two authors are supported by the Austrian Science Fund (FWF), project Y-901.}

\maketitle

Let $(x_n)_{n \geq 1}$ be a sequence of real numbers. We say that this sequence is \emph{equidistributed} or \emph{uniformly distributed modulo one} if asymptotically the relative number of fractional parts of elements of the sequence falling into a certain subinterval is proportional to the length of this subinterval. More precisely, we require that
$$
\lim_{N \to \infty} ~\frac{1}{N} ~\# \Big\{ 1 \leq n \leq N:~\{ x_n \} \in [a,b] \Big\} = b-a
$$
for all $0 \leq a \leq b \leq 1$, where $\{ \cdot\}$ denotes the fractional part. This notion was introduced in the early twentieth century, and received widespread attention after the publication of Hermann Weyl's seminal paper \emph{\"Uber die Gleich\-ver\-teilung von Zahlen mod. Eins} in 1916 \cite{weyl}. Among the most prominent results in the field are the facts that the sequences $(n \alpha)_{n \geq 1}$ and $(n^2 \alpha)_{n \geq 1}$ are equidistributed whenever $\alpha \not\in \Q$, and the fact that for any distinct integers $n_1, n_2, \dots$ the sequence $(n_k \alpha)_{k \geq 1}$ is equidistributed for almost all $\alpha$. We note that when $(X_n)_{n \geq 1}$ is a sequence of independent, identically distributed (i.i.d.) random variables having uniform distribution on $[0,1]$, then by the Glivenko--Cantelli theorem this sequence is almost surely equidistributed. Consequently, in a very vague sense equidistribution can be seen as an indication of ``pseudorandom'' behavior of a deterministic sequence. For more information on uniform distribution theory, see the monographs \cite{dts,knu}.\\

The investigation of pair correlations can also be traced back to the beginning of the twentieth century, when such quantities appeared in the context of statistical mechanics. In our setting, when $(x_n)_{n \geq 1}$ are real numbers in the unit interval, we define a  function $F_N(s)$ as 
\begin{equation} \label{fdef}
F_N(s) = \frac{1}{N} \left\{1 \leq m, n \leq N,~m \neq n:~\|x_m - x_n\| \leq \frac{s}{N} \right\},
\end{equation}
and
$$
F(s) = \lim_{N \to \infty} F_N(s),
$$
provided that such a limit exists; here $s \geq 0$ is a real number, and $\| \cdot \|$ denotes the distance to the nearest integer. The function $F_N(s)$ counts the number of pairs $(x_m,x_n),~1 \leq m,n \leq N, ~m \neq n,$ of points which are within distance at most $s/N$ of each other (in the sense of the distance on the torus). If $F(s) = 2s$ for all $s \geq 0$, then we say that the asymptotic distribution of the pair correlations of the sequence is \emph{Poissonian}. Again, one can show that an i.i.d.\ random sequence (sampled from the uniform distribution on $[0,1]$) has this property, almost surely. Questions concerning the distribution of pair correlations of sequences such as $(\{n \alpha\})_{n \geq 1}$ or $(\{n^2 \alpha\})_{n \geq 1}$ are linked with statistical properties of the distribution of the energy spectra of certain quantum systems, and thus play a role for the Berry--Tabor conjecture \cite{bt}. See \cite{mark,rs} for more information on this connection. It turns out that unlike the situation in the case of equidistribution, there is \emph{no} value of $\alpha$ for which the sequence $(\{n \alpha\})_{n \geq 1}$ has Poissonian pair correlations, and the question whether the pair correlations of the sequence $(\{n^2 \alpha\})_{n \geq 1}$ are Poissonian or not depends on delicate number-theoretic properties of $\alpha$, in particular on properties concerning Diophantine approximation and the continued fraction expansion of $\alpha$. Here many problems are still open \cite{heath,rsz}. Furthermore, for $(n_k)_{k \geq 1}$ being distinct integers the question whether $(\{n_k \alpha\})_{k \geq 1}$ has Poissonian pair correlations for almost all $\alpha$ or not depends on certain number-theoretic properties of $(n_k)_{k \geq 1}$, in particular on the number of possible ways to represent integers as a difference of elements of this sequence; see \cite{all}.\\

It is remarkable that (to the best of our knowledge) the relation between these two notions (being equidistributed, and having Poissonian pair correlations) has never been clarified, a fact which came to our attention by a question asked by Arias de Reyna \cite{arias} in a slightly different, but related context (we will repeat this question at the end of the present section). As a starting observation, we note that in a probabilistic sense Poissonian pair correlations actually require uniform distribution. More precisely, assume that $(X_n)_{n \geq 1}$ are i.i.d.\ random variables, which for simplicity we assume to have a density $g(x)$ on $[0,1]$. Then we have
\begin{eqnarray}
& & \mathbb{E} \left( \frac{1}{N} \left\{1 \leq m, n \leq N,~m \neq n:~\|X_m - X_n\| \leq \frac{s}{N} \right\} \right) \nonumber\\
& \approx & \frac{1}{N} ~N^2 \int_0^1 g(x) \underbrace{\int_{x-\frac{s}{N}}^{x+\frac{s}{N}} g(y) ~dy}_{\approx \frac{2s}{N} g(x)} ~dx \nonumber\\
& \approx & 2s \int_0^1 g(x)^2~dx, \nonumber
\end{eqnarray}
which can be turned into a rigorous argument to show that almost surely
\begin{equation} \label{fs2}
F(s) = 2 s \int_0^1 g(x)^2~dx, \qquad s > 0,
\end{equation}
in this case (and $F(s)=\infty$ for all $s>0$, almost surely, in the case when the distribution of the $X_n$ is not absolutely continuous with respect to the Lebesgue measure). Now we clearly have $\int_0^1 g(x)^2~dx = 1$ if and only if $g(x) \equiv 1$, which means that $g$ is the density of the uniform distribution. Thus Poissonian pair correlations require uniform distribution in a probabilistic sense; however, it is a priori by no means clear that a similar relation also holds for the case of deterministic sequences $(x_n)_{n \geq 1}$. Our Theorem \ref{th1} below shows that this actually is the case.

\begin{theorem} \label{th1}
Let $(x_n)_{n \geq 1}$ be a sequence of real numbers in $[0,1]$, and assume that the distribution of its pair correlation statistic is asymptotically Poissonian. Then the sequence is equidistributed.
\end{theorem}

There also is a quantitative ``density'' version of the theorem, which has a resemblance of \eqref{fs2}, and which we state as Theorem \ref{th2} below.

\begin{theorem} \label{th2}
Let $(x_n)_{n \geq 1}$ be a sequence of real numbers in $[0,1]$. Assume that it has an asymptotic distribution function $G(x)$ on $[0,1]$, i.e., that there is a function $G(x)$ such that
$$
G(x) = \lim_{N \to \infty} \frac{1}{N} \# \Big\{ 1 \leq n \leq N:~x_n \in [0,x] \Big\}, \qquad x \in [0,1].
$$
Assume also that there is a function $F(s):~[0,\infty) \mapsto [0,\infty]$ such that
$$
F(s) = \lim_{N \to \infty} \frac{1}{N} \left\{1 \leq m, n \leq N,~m \neq n:~\|x_m - x_n\| \leq \frac{s}{N} \right\}, \qquad \text{$s > 0$}.
$$
Then the following hold. 
\begin{itemize}
 \item If $G$ is not absolutely continuous, then $F(s) = \infty$ for all $s > 0$.
 \item If $G$ is absolutely continuous, then, writing $g$ for the density function of the correpsonding measure, we have
\begin{equation} \label{fs22}
\limsup_{s \to \infty} \frac{F(s)}{2s} \geq \int_0^1 g(x)^2~dx.
\end{equation}
\end{itemize}
\end{theorem}

We believe that Theorem \ref{th1} is quite remarkable; actually, we initially set out to prove its opposite, namely that a sequence which has Poissonian pair correlations does not have to be equidistributed. This seemed natural to us since equidistribution is controlled by the ``large-scale'' behavior, while pair correlations are determined by ``fine-scale'' behavior. Only after some time we realized why it is not possible to construct a non-equidistributed sequence which has Poissonian pair correlations; roughly speaking, the reason is that regions where too many points are situated contribute to the pair correlation function proportional to the square of the local density, and regions with fewer elements cannot compensate this larger contribution -- this is exactly what \eqref{fs2} and \eqref{fs22} also tell us.\\

There is a famous characterization of equidistribution in terms of exponential sums, called \emph{Weyl's criterion}. In a similar way one could characterize the asymptotic pair correlation function by exponential sums, and then assuming Poissonian pair correlations try to control the exponential sums in Weyl's criterion. However, we have not been able to do this; the problem is of course that the pair correlations are determined by ``fine'' properties at the scale of $1/N$, while equidistribution is a ``global'' property on full scale -- in other words, the trigonometric functions which dertermine the distribution of the pair correlations have frequencies of order $N$, while equidistribution is determined by trigonometric functions with constant frequencies. Instead of following such an approach, our proof of Theorem \ref{th1} is based on estimating the pair correlation function by a certain quadratic form, which is attached to a so-called \emph{circulant matrix}. We can calculate the eigenvalues and eigenvectors of this matrix, and after averaging over different values of $s$ reduce the problem to the fact that the Fej\'er kernel is a non-negative kernel.\\

Finally we return to the question of Arias de Reyna, which was mentioned above. Elliott and Hlawka independently proved that the imaginary parts $(\gamma_n)_{n \geq 1}$ of the non-trivial zeros of the Riemann zeta function are equidistributed. However, the proof of this result is simplified by the fact that the zeros of the zeta function are relatively dense; more precisely, the number of zeros up to height $T$ is roughly $\frac{T}{2 \pi} \log \left(\frac{T}{2\pi}\right) - \frac{T}{2 \pi}$. Thus to get a statement about the pseudorandomness of these zeros it is more interesting to consider the sequence of imaginary parts of zeros after normalizing them to have average distance one; that is, instead of investigating the equidistribution of the sequence $(\gamma_n)_{n \geq 1}$ itself one asks for the equidistribution of the normalized sequence $(x_n)_{n \geq 1} = \left(\frac{\gamma_n}{2 \pi} \log \left(\frac{\gamma_n}{2\pi e}\right) \right)_{n \geq 1}$. This seems to be a very difficult problem; see \cite{arias} for the current state of research in this direction. On the other hand, the famous \emph{Montgomery pair correlations conjecture} predicts a certain asymptotic distribution $R(s)$ for the pair correlations between elements of this normalized sequence $(x_n)_{n \geq 1}$. For the statement of this conjecture see \cite{mont}; we only mention that the distribution is not the same as in the case of a random sequence, but coincides with a distribution that also appears as the correlation function of eigenvalues of random Hermitian matrices and shows a certain ``repulsion'' phenomenon. Arias de Reyna asked whether Montgomery's pair correlation conjecture is compatible with equidistribution of the normalized zeros.\\

Note that the setting of this question is different from our setting; while in our setting the whole sequence is contained in $[0,1]$ and the average spacing of the first $N$ points is $1/N$, in the setting of Arias de Reyna's question the equidistribution property is requested for the reduction of the sequence $(x_n)_{n \geq 1}$ modulo one, while the pair correlations are calculated for the increasing sequence $(x_n)_{n \geq 1}$ itself, for which the average spacing between consecutive elements is 1 in the limit. Thus the results from the present paper cannot be applied to this setting. A general form of Arias de Reyna's question is: Let $(x_n)$ be an increasing sequence with average spacing 1, that is, $x_n/n \to 1$. Assume that $(x_n)_{n \geq 1}$ asymptotically has the pair correlation distribution $R(s)$ from Montgomery's conjecture. Is it possible that $(x_n)_{n \geq 1}$ is equidistributed? Is it possible that $(x_n)_{n \geq 1}$ is \emph{not} equidistributed? It is known that there exists a random process whose pair correlation function is $R(s)$ (see \cite{kls,kls2}), and one should be able to show that such a random process (or a further randomization of it) is equidistributed almost surely. So the answer to the first question is ``yes''. The answer to the second question should be ``yes'' as well, but we have not been able to construct an example.

\section{Preliminaries}

Throughout this section, we will use the following notation. Assume that $x_1, \dots, x_N$ are given. Let $F_N(s)$ be defined as in \eqref{fdef}. We partition the unit interval $[0,1)$ into subintervals $I_1, \dots, I_M$, where $I_m = [m/M,(m+1)/M)$, and we set
$$
y_m = \# \Big\{ 1 \leq n \leq N:~x_n \in I_m \Big\}.
$$
Then trivially we have
$$
\sum_{m=1}^M y_m = N.
$$
For notational convenience, we assume that the sequence $(y_m)_{1 \leq m \leq M}$ and the partition $I_1, \dots, I_M$ are extended periodically; in other words, we set
$$
y_m = y_{(m~\textup{mod}~M)}, \qquad \textrm{and} \qquad I_m = I_{(m~\textup{mod}~M)}, \qquad m \in \mathbb{Z}.
$$
Let $s \geq 1$ be an integer. We set
$$
H_{N,M} (s) = \sum_{m=1}^M~ \sum_{-s+1 \leq \ell \leq s-1} y_m y_{m + \ell}.
$$
Then by construction we have
\begin{eqnarray} 
H_{N,M}(s) & = & \sum_{m=1}^M ~~\sum_{n \in \{1, \dots, N\}:~x_n \in I_m}~ \# \left\{ 1 \leq k \leq N:~x_k \in \bigcup_{\ell=-s+1}^{s-1} I_{m+\ell} \right\} \nonumber\\
& \leq & \sum_{n=1}^N \left\{1 \leq k \leq N:~\|x_k - x_n\| \leq \frac{s}{M} \right\} \nonumber\\
& = & \left( \sum_{n=1}^N \left\{1 \leq k \leq N, ~n \neq k:~\|x_k - x_n\| \leq \frac{s}{M} \right\} \right) + N  \nonumber\\
& = & NF_N\left( \frac{sN}{M} \right) + N. \label{hnfn}
\end{eqnarray}
Thus a lower bound for $H_{N,M}$ implies a lower bound for $F_N$.\\

We have the following lemma.

\begin{lemma} \label{lemma1}
Let $y_1, \dots, y_M$ be non-negative real numbers whose sum is $N$, assume that $(y_m)_{1 \leq m \leq M}$ is extended periodically as above, and let $H_{N,M}(s)$ be defined as above. Let $S \geq 1$ be an integer for which $2S < M$. Then
$$
\frac{1}{S} \sum_{s=1}^S H_{N,M}(s) \geq \frac{SN^2}{M}.
$$
\end{lemma}

\begin{proof}
The sum 
$$
\sum_{m=1}^M ~\sum_{-s+1 \leq \ell \leq s-1} y_m y_{m + \ell}
$$
in the definition of $H_{N,M}$ is a quadratic form which is attached to the matrix
$$
A^{(s)} = \left(a_{ij}^{(s)}\right)_{1 \leq i,j \leq M} = \left\{ \begin{array}{ll} 1 & \text{if $\textup{dist}(i-j) \leq s-1$,}\\ 0 & \text{otherwise,} \end{array} \right.
$$
where $\textup{dist}$ is the periodic distance such that $\textup{dist}(i-j) \leq s-1$ whenever
$$
i-j ~\in ~(-\infty,-M+s-1] \cup [-s+1,s-1] \cup [M-s+1,\infty).
$$
Thus $A^{(s)}$ is a band matrix which also has non-zero entries in its right upper and left lower corner. This matrix $A^{(s)}$ is symmetric, and it is of a form which is called \emph{circulant}. Generally, a circulant matrix is a matrix of the form
$$
\begin{pmatrix}
c_0 & c_1 & c_2 & \dots & c_{M-1} \\
c_{M-1} & c_0 & c_1 & \dots & c_{M-2} \\
c_{M-2} & c_{M-1} & c_0 & \ddots & c_{M-3} \\
\vdots & & \ddots & \ddots & \vdots \\
c_1 & & \dots & c_{M-1} & c_0 
\end{pmatrix},
$$
where each row is obtained by a cyclic shift of the previous row. We recall some properties of such matrices; for a reference see for example \cite[Chapter 3]{gray}. The eigenvectors of such a matrix are 
\begin{equation} \label{vk}
v_m = \left(1, \omega^m, \omega^{2m}, \dots, \omega^{(M-1)m} \right), \qquad m = 0, \dots, M-1,
\end{equation}
where $\omega = e^{\frac{2 \pi i}{M}}$. Note that these eigenvectors are pairwise orthogonal, and that they are independent of the coefficients of the matrix (they just depend on the fact that the matrix is circulant). The eigenvalue $\lambda_m$ to the eigenvector $v_m$ is given by
\begin{equation} \label{eigenv}
\lambda_m = \sum_{\ell=0}^{M-1} c_\ell \omega^m.
\end{equation}

We have already noted that our matrix $A^{(s)}$ is symmetric, which implies that all its eigenvalues are real. Furthermore, if we use the formula \eqref{eigenv} to calculate the eigenvalues of $A^{(s)}$ then we obtain
$$
\lambda_m^{(s)} = \sum_{\ell=-s+1}^{s-1} \omega^m = \frac{\sin \left(\frac{(2s-1)\pi m}{M} \right)}{\sin\left(\frac{\pi m}{M}\right)},
$$
which is the $s-1$-st order Dirichlet kernel $D_{s-1}$ (with period $1$ rather than the more common period $2 \pi$), evaluated at position $m/M$. Note that the largest eigenvalue is $\lambda_0^{(s)} = 2s-1$.\\

Since the eigenvectors of $A^{(s)}$ form an orthogonal basis, we can express our vector $(y_1, \dots, y_M)$ in this basis. We write
$$
(y_1, \dots, y_M) = \sum_{m=0}^{M-1} \varepsilon_m v_m
$$ 
for appropriate coefficients $(\varepsilon_m)_{1 \leq m \leq M}$. Note that we have $y_1+\dots+y_M=N$, which can be rewritten as $(y_1, \dots, y_M) v_0 = N$; thus we must have $\varepsilon_0 = N/M$ (since the eigenvectors are orthogonal). Furthermore, we have
\begin{eqnarray} 
H_{N,M}(s) & = & \left( \sum_{m=0}^{M-1} \varepsilon_m v_m \right)^T A^{(s)} \left( \sum_{m=0}^{M-1} \varepsilon_m v_m \right) \nonumber\\
& = & \sum_{m=0}^{M-1} \lambda_m^{(s)} \varepsilon_m^2 \|v_m\|_2^2 \nonumber\\
& = & M \sum_{m=0}^{M-1} \lambda_m^{(s)} \varepsilon_m^2, \label{epsin}
\end{eqnarray}
again by orthogonality. However, from this we cannot deduce that $H_{N,M}(s) \geq M \lambda_0^{(s)} \varepsilon_0^2 = (2s-1) N^2/M$, since (in general) some of the eigenvalues are negative. To solve this problem we will make a transition from the Dirichlet kernel to the Fej\'er kernel, which is non-negative.\\

We repeat that the eigenvectors of $A^{(s)}$ depend on $M$, but not on $s$. Let $S \geq 1$ be an integer and consider
$$
A^{(\Sigma)} = \frac{1}{S} \sum_{s=1}^S A^{(s)},
$$
where we assume that $S < 2M$ (to retain the structure of the matrix). Then clearly the eigenvectors of this matrix are also given by $v_0, \dots, v_{M-1}$, and the corresponding eigenvalues are
$$
\lambda_m^{(\Sigma)} = \frac{1}{S} \sum_{s=1}^{S} \lambda_m^{(s)} = \frac{1}{S} \sum_{s=1}^{S} ~\sum_{\ell=-s+1}^{s-1} \omega^m, \qquad 0 \leq m \leq M-1.
$$
Now $\lambda_m^{(\Sigma)}$ can be identified as the Fej\'er kernel of order $S-1$ (with period $1$ instead of $2 \pi$), evaluated at position $m/M$. It is well-known that the Fej\'er kernel is non-negative, so we have
\begin{equation} \label{non-neg}
\lambda_m^{(\Sigma)} \geq 0, \qquad m = 0, \dots, M-1,
\end{equation}
and we also have
$$
\lambda_0^{(\Sigma)} = \frac{1}{S} \sum_{s=1}^{S} (2s-1) = S.
$$
Now using again the considerations which led to \eqref{epsin} we can show that
$$
\frac{1}{S} \sum_{s=1}^S H_{N,M}(s) \geq M \sum_{m=0}^{M-1} \lambda_m^{(\Sigma)} \varepsilon_m^2 \geq M \lambda_0^{(\Sigma)} \varepsilon_0^2  = SN^2/M,
$$
where \eqref{non-neg} played a crucial role. This proves the lemma.
\end{proof}

\section{Proof of Theorem \ref{th1}}

Let $(x_n)_{n \geq 1}$ be a sequence of real numbers in $[0,1]$, and assume that it is \emph{not} equidistributed. Thus there exists an $a \in (0,1)$ for which 
$$
\frac{1}{N} \sum_{n=1}^N \mathds{1}_{[0,a)} (x_n) \not\to a \qquad \text{as $N \to \infty$}
$$ 
(here, and in the sequel, $\mathds{1}_B$ denotes the indicator function of a set $B$). However, for this value of $a$ by the Bolzano--Weierstra{\ss} theorem there exists a subsequence $(N_r)_{r \geq 1}$ of $\mathbb{N}$ along which a limit exists; that is, there exists a number $b \neq a$ such that
\begin{equation} \label{limr}
\lim_{r \to \infty} \frac{1}{N_r} \sum_{n=1}^{N_r} \mathds{1}_{[0,a)} (x_n) = b.
\end{equation}

Let $\varepsilon >0$ be given, and assume that $\varepsilon$ is ``small''. Choose an integer $S$ (which is ``large''). Let $r \geq 1$ be given, let $N_r$ be from the subsequence in the previous paragraph, and consider the points $x_1, \dots, x_{N_r}$. Let $\mathcal{E}$ denote the union of the sets
$$
\left[0,\frac{2S}{N_r} \right] \cup \left[a - \frac{2S}{N_r}, a + \frac{2S}{N_r} \right] \cup \left[1-\frac{2S}{N_r},1 \right].
$$
Furthermore, we set $B_1 = [0,a] \backslash \mathcal{E}$ and $B_2=[a,1] \backslash \mathcal{E}$.\\

First consider the case that $\# \left\{1 \leq n \leq N_r:~ x_n \in \mathcal{E} \right\} \geq \varepsilon N_r$. Then by the pigeon hole principle there exists an interval of length at most $1/N_r$ in $\mathcal{E}$ which contains at least $\varepsilon N_r/(8S)$ elements of $\{x_1, \dots, x_{N_r}\}$. All of these numbers are within distance $1/N_r$ of each other, which implies that
$$
N_r F_{N_r}(1) \geq \left(\frac{\varepsilon N_r}{8S} \right)^2 - N_r.
$$
If this inequality holds for infinitely many $r$, then
$$
\limsup_{r \to \infty} F_{N_r}(1) = \infty,
$$
which implies that the pair correlations distribution cannot be asymptotically Poissonian.\\

Thus we may assume that $\# \left\{1 \leq n \leq N:~ x_n \in \mathcal{E} \right\} < \varepsilon N_r$ for all elements of the subsequence $(N_r)_{r \geq 1}$. Then $[0,1] \backslash \mathcal{E} = B_1 \cup B_2$ contains at least $(1-\ve) N_r$ elements of $\{x_1, \dots, x_{N_r}\}$. Consequently, if $r$ is sufficiently large, by \eqref{limr} we have
$$
\# \{1 \leq n \leq N_r: x_n \in B_1\} \geq (b-2\ve) N_r,
$$
and
$$
\# \{1 \leq n \leq N_r: x_n \in B_2\} \geq ((1-b)-2\ve) N_r.
$$
We assume that $r$ is so large that we can find positive integers $M_1, M_2$ for which $a/M_1 \approx (1-a)/M_2 \approx \frac{1}{N_r}$; more precisely, we demand that
\begin{equation} \label{dem}
\frac{a}{M_1} \in \left[\frac{1-\ve}{N_r}, \frac{1}{N_r} \right], \qquad \frac{1-a}{M_2} \in \left[ \frac{1-\ve}{N_r}, \frac{1}{N_r} \right].
\end{equation}
We partition $B_1$ and $B_2$ into $M_1$ and $M_2$ disjoint subintervals of equal length, respectively, and write $y_1, \dots, y_{M_1}$ and $z_1, \dots, z_{M_2}$ for the number of elements contained in each of these subintervals (we assume that the subintervals are sorted in the ``natural'' order from left to right). Next, for $s \in \{1, \dots, S\}$ we define
$$
H_{M_1}^* (s) = \sum_{m=1}^{M_1} ~\sum_{-s+1 \leq \ell \leq s-1} y_m y_{m + \ell}
$$
and
$$
H_{M_2}^* (s) = \sum_{m=1}^{M_2} ~\sum_{-s+1 \leq \ell \leq s-1} z_m z_{m + \ell}.
$$
By construction, $\sum_{m=1}^{M_1} y_m \geq (b-2\ve) N_r$ and $\sum_{m=1}^{M_2} z_m \geq ((1-b)-2\ve)N$. Also by construction the cyclic extension is not necessary here, provided that $r$ is sufficiently large; by excluding all the points in $\mathcal{E}$ we have $y_1= \dots = y_S = 0$ and $y_{M_1-S+1} = \dots = y_{M_1}= 0$, and the same holds for the $z_m$'s.\\

Then by Lemma \ref{lemma1} and by our choice of $M_1,M_2$ we have
\begin{equation} \label{h1*}
\frac{1}{S} \sum_{s=1}^S H_{M_1}^* (s) \geq \frac{S ((b-2\ve)N)^2}{M_1} \geq \frac{S (b-2\ve)^2 N_r (1-\ve)}{a},
\end{equation}
and
\begin{equation} \label{h2*}
\frac{1}{S} \sum_{s=1}^S H_{M_2}^* (s) \geq \frac{S ((1-b-2\ve)N)^2}{M_2} \geq \frac{S (1-b-2\ve)^2 N_r (1-\ve)}{1-a}.
\end{equation}
As in the calculation leading to \eqref{hnfn} we can obtain a lower bound for the pair correlation function $F_{N_r}(s)$ from the lower bounds for $H_{M_1}^*$ and $H_{M_2}^*$. More precisely, we obtain
$$
N_r F_{N_r} \left(s\right) + N_r \geq H_{M_1}^*(s) + H_{M_2}^*(s),
$$
and accordingly, by \eqref{h1*} and \eqref{h2*}, we have
\begin{eqnarray}
& & \frac{1}{S} \sum_{s=1}^S \left( N_r F_{N_r} (s) + N_r \right) \nonumber\\
& \geq & (1 - \ve) S N_r \left(\frac{(b-\ve)^2}{a}+ \frac{(1-b-\ve)^2}{1-a}\right). \label{impl}
\end{eqnarray}
Now note that for $0 \leq a,b \leq 1$ we can only have $b^2/a + (1-b)^2/(1-a)=1$ if $a=b$; however, this is ruled out by assumption. For all other pairs $(a,b)$ we have $b^2/a + (1-b)^2/(1-a)>1$, and thus \eqref{impl} implies that
$$
\frac{1}{S} \sum_{s=1}^S N_r F_{N_r}(s) \geq N_r \left(S (1+2c_{\ve}) -1\right)
$$
for a positive constant $c_{\ve}$ depending only on $\ve$, provided that $\varepsilon$ is sufficiently small. This implies
\begin{equation} \label{fnse}
\frac{1}{S} \sum_{s=1}^S F_{N_r} (s) \geq S (1+2c_{\ve}) - 1 \geq S (1+c_{\ve}) \left(1 + \frac{1}{S} \right),
\end{equation}
where the last inequality holds under the assumption that $S$ is sufficiently large. Consequently there exists an $s \in \{1,\dots, S\}$ such that
\begin{equation} \label{fnr}
F_{N_r}(s) \geq (1+c_{\ve}) 2s,
\end{equation}
since otherwise \eqref{fnse} is impossible.\\

For every sufficiently large $N_r$ in the subsequence in \eqref{limr} such an $s \in \{1, \dots, S\}$ exists; accordingly, there is an $s$ such that for infinitely many $r$ we have \eqref{fnr}. Thus for this $s$ we have
$$
\limsup_{r \to \infty} \frac{F_{N_r}(s)}{2s} \geq (1+c_{\ve}) > 1,
$$
which proves the theorem.

\section{Proof of Theorem \ref{th2}}

First assume that the measure $\mu_G$ defined by the asymptotic distribution function $G(x)$ is not absolutely continuous with respect to the Lebesgue measure. A function which is not absolutely continuous is not Lipschitz continuous as well. Thus there is an $\varepsilon >0$ such that for every $\delta>0$ there exists an interval $I \subset [0,1]$ such that
$$
\lambda(I) \leq \delta, \qquad \text{but} \qquad \mu_G(I) \geq \ve,
$$
where $\lambda$ denotes the Lebesgue measure (that is, the length) of $I$. Let $\hat{I}$ denote the interval $I$ after removing a subinterval of length $N^{-1}$ from the left and right end, respectively (to remove the influence of the cyclic ``overlap'' in Lemma \ref{lemma1}). Then for sufficiently large $N$ the interval $\hat{I}$ contains at least $\ve N/2$ elements of $(x_n)_{1 \leq n \leq N}$. Set $M = \lceil \lambda(I) N \delta^{-1/2} \rceil$, split $I$ into $M$ subintervals, and denote the number of elements of the set $\{x_1, \dots, x_N\} \cap \hat{I}$ contained in each of these subintervals by $y_1, \dots, y_M$, respectively. Let $\hat{N} = y_1 + \dots + y_M$, and define $H_{M,\hat{N}}(1) = y_1^2 + \dots + y_M^2$. Applying Lemma \ref{lemma1} and using a rescaled version of \eqref{hnfn} we have
\begin{eqnarray*}
N F_N \left(2 \delta^{1/2} \right) + N & \geq &  N F_N \left(\frac{2 N \lambda(I)}{M} \right) + N \\
& \geq & H_{M,\hat{N}} (1) \\
& \geq & \frac{(\ve N/2)^2}{M} \\
& \geq & \frac{\ve^2 N}{5 \delta^{1/2}}
\end{eqnarray*}
for sufficiently large $N$. Since $\ve$ is fixed and $\delta$ can be chosen arbitrarily small, this proves the theorem when $\mu_G$ is not absolutely continuous.\\

Now assume that the measure $\mu_G$ defined by $G(x)$ is absolutely continuous with respect to the Lebesgue measure, and thus has a density $g(x)$. In the sequel we will think of $([0,1],\mathcal{B}([0,1]), \lambda)$ as a probability space, and write $\mathbb{E}$ for the expected value (of a measurable real function) with respect to this space. We split the unit interval into $2^R$ intervals of equal lengths. Let $\mathcal{F}_R$ denote the $\sigma$-field generated by these intervals. Assume for simplicity that $g$ is bounded on $[0,1]$. Then we can use arguments similar to those above to prove that for given $\ve>0$ there exist infinitely many values of $s$ such that for each of these values we have
$$
\frac{F\left(s\right)}{2s} \geq \mathbb{E} \left( \left(\mathbb{E} \big( g | \mathcal{F}_R \big) \right)^2 \right) - \ve,
$$
where $\mathbb{E} \left( g | \mathcal{F}_k \right)$ denotes the conditional expectation of $g$ under the $\sigma$-field $\mathcal{F}_R$.\footnote{In the proof of Theorem \ref{th1}, the role of the second moment of the conditional expectation function is played by the expression $b^2/a + (1-b)^2/(1-a)$, which appears in line \eqref{impl}.} Note that a direct generalization of the proof of Theorem \ref{th1} only guarantees the existence of \emph{one} such integer $s$; however, we can use the fact that $F(s)$ is monotonically increasing to show that there actually must be infinitely many such values of $s$. The family $(\mathcal{F}_R)_{R \geq 1}$ forms a filtration whose limit is $\mathcal{B}([0,1])$, in the sense that $\mathcal{B}([0,1])$ is the sigma-field generated by $\bigcup_{R \geq 1} \mathcal{F}_R$. Thus by the convergence theorem for conditional expectations (also known as \emph{L\'evy's zero-one law}) we have
$$
\lim_{R \to \infty} \mathbb{E} \left( \left(\mathbb{E} \big( g | \mathcal{F}_R \big) \right)^2 \right) = \mathbb{E} \left(g^2 \right) = \int_0^1 g(x)^2~dx,
$$
which proves the theorem in the case when $g$ is bounded. Finally, if $g$ is not bounded then we can apply the argument above to a truncated version $g_{\textup{trunc}}$ of $g$ and show that in this case
$$
\limsup_{s \to \infty} \frac{F(s)}{2s} \geq \int_0^1 g_{\textup{trunc}}(x)^2~dx.
$$
By raising the level where $g$ is truncated this square-integral can be made arbitrarily close to $\int_0^1 g(x)^2~dx$, or arbitrarily large in case we have $\int_0^1 g(x)^2~dx= \infty$. This proves the theorem.

\section*{Acknowledgements}

We want to thank Ivan Izmestiev (University of Fribourg) for his very helpful comments during the preparation of this manuscript.


\end{document}